\documentclass[11pt]{article}   
\usepackage{amssymb,amscd,latexsym}   
\usepackage{amsmath}
\usepackage{amsthm}
\usepackage{mathdots}
\usepackage[pagebackref]{hyperref}
\usepackage{color}
\usepackage[all]{xy}
\newcommand{\rar}{\rightarrow}
\newcommand{\lar}{\longrightarrow}
\newcommand{\llar}{-\kern-5pt-\kern-5pt\longrightarrow}

\newtheorem{Theorem}{Theorem}[section]
\newtheorem{Lemma}[Theorem]{Lemma}
\newtheorem{Corollary}[Theorem]{Corollary}
\newtheorem{Proposition}[Theorem]{Proposition}

\newtheorem{Definition}[Theorem]{Definition}

\def\sqr#1#2{{\vcenter{\hrule height.#2pt
			\hbox{\vrule width.#2pt height#1pt \kern#1pt
				\vrule width.#2pt}
			\hrule height.#2pt}}}
\def\phi{\varphi}

\DeclareMathOperator{\Spec}{Spec}

\DeclareMathOperator{\rank}{rank}
\DeclareMathOperator{\Ht}{ht}

\def\xx{{\bf x}}

\def\tt{{\bf t}}

\def\fm{{\mathfrak m}}
\def\fn{{\mathfrak n}}

\def\Ht{{\rm ht}\,}

\def\restr{{\kern-1pt\restriction\kern-1pt}}

\def\pp{{\mathbb P}}

\begin{document}
	\begin{center}
		{\Large{\bf\sc   Additions to a  theorem of Morey-Ulrich}}
		\footnotetext{AMS Mathematics
			Subject Classification (2010   Revision). Primary 13A02, 13A30, 13D02, 13H10, 13H15; Secondary  14E05, 14M07, 14M10,  14M12.} 
		\footnotetext{	{\em Key Words and Phrases}:  codimension two perfect ideal, Rees algebra, special fiber,  Cohen--Macaulay, Jacobian dual, expected form, fiber type.}
		
		\vspace{0.1true in}

		\vspace{0.3in}
		{\large\sc Thiago Fiel}\footnote{Supported by a CNPq grant (171302/2023-0)} \quad
	 \quad	
	{\large\sc Zaqueu Ramos}\footnote{Partially
			supported by a CNPq grant (304122/2022-0)} \quad
		{\large\sc Aron  Simis}\footnote{Partially
			supported by a CNPq grant (301131/2019-8).}

	\end{center}
	
	\begin{abstract}
	Let $R = k[x_1,\ldots, x_d] (d\geq 3)$ denote a standard graded polynomial ring over 
	an algebraically closed field $k$, and let $I \subset R$ be a perfect ideal of codimension $2$ with an $n\times (n-1)$ linear presentation matrix. We prove an extended formulation of a theorem of Morey and Ulrich in the case where $I$ satisfies condition $G_{d-1}$, but not condition $G_d$.
	\end{abstract}
	

	\section{Statement of the problem}\label{overture}
	
	Let $R = k[x_1,\ldots, x_d] (d\geq 3)$ denote a standard graded polynomial ring over 
	an infinite field $k$, and let $I \subset R$ be a perfect ideal of codimension $2$ with an $n\times (n-1)$ linear presentation matrix $\phi$, commonly referred to as the (associated) {\em  Hilbert--Burch matrix.}
	Thus, $I$ is generated by the $n$ maximal minors of $\phi$.
	Here we adhere to the habit that, over a Cohen--Macaulay base ring, ``codimension'' and ``height'' are used interchangeably.

Now, there is a unique $d\times(n-1)$-matrix $B$, whose  entries are linear forms in $k[\tt]=k[t_1,\ldots,t_n]$, such that 
	\begin{equation}
		\tt\cdot\varphi=\xx\cdot B.
	\end{equation} 
	$B$ is often called the {\em Jacobian dual matrix} of $\phi$.
	
	S. Morey and B. Ulrich have proved the following beautiful theorem:
	
	\begin{Theorem}\label{MainM-U}{\rm (\cite[Theorem 1.3]{MU})} 
		Let $R$ and $I\subset R=k[x_1,\ldots,x_d]$ be as above, and such that $\mu(I) \geq d+1$. {\sc If $I$ satisfies condition $G_d$}, then$:$
		\begin{enumerate}
			\item[\rm(i)]  The analytic spread $\ell(I)$ is maximal {\rm (}i.e., $\ell(I)=d${\rm )}.
			\item[\rm(ii)]  The reduction number of $I$ has the      {\rm ``}expected\,{\rm ''}  value  {\rm (}i.e., $\ell(I)-1${\rm )}.
			\item[\rm(iii)] The Rees algebra $\mathcal{R}_R(I)$ of $I$ is Cohen--Macaulay with  homogeneous defining ideal of the {\rm ``}expected\,{\rm ''} form $\langle I_1(\tt\cdot\varphi), I_d(B)\rangle$.
		\end{enumerate}
	\end{Theorem}
Some authors (\cite{CPW1}, \cite{DRS}, \cite{L1}, \cite{L2}) have considered a weakening of the main hypothesis of this theorem, by replacing it with the assumption that $I$ satisfies condition $G_{d-1}$, but not condition $G_d$, obtaining similar results.
Let us refer to this condition as $(*)$ just for the sake of this introduction and to avoid repetition.

For instance, \cite[Theorem 5.2]{L1} states that, under $(*)$, if $k$ is algebraically closed, $d\geq 3$  arbitrary, and $n=d+1,$ then the defining ideal of $\mathcal{R}_R(I)$ is $\langle I_{1}(\tt\cdot \phi), \det B'\rangle$ where $B'$ is a certain $(d-1)\times (d-1)$ submatrix of the Jacobian dual matrix $B.$

Further, \cite[Theorem 4.6]{L2} -- recovered by different methods in \cite[Theorem 3.6]{DRS} -- states that, under $(*)$, if $k$ is algebraically closed, $d=3$, $n\geq d+1$, and moreover, after a change of variables, the matrix $\phi$ has rank  $1$ modulo the ideal generated by two variables, then  the defining ideal of $\mathcal{R}_R(I)$ is $\langle I_{1}(\tt\cdot \phi), I_2(B')\rangle$ where $B'$ is a certain  $2\times (n-2)$ submatrix of the Jacobian dual matrix $B.$

Still further, \cite[Theorem 3.9]{CPW1} states that, under $(*)$, if $k$ is an arbitrary field, $d\geq 3$, $n\geq d+1$, and upon a change of variables, the matrix $\phi$ has rank  $1$ modulo the ideal $\langle x_1,\ldots,x_{d-1}\rangle,$ then the defining ideal of $\mathcal{R}_R(I)$ is $\langle I_{1}(\tt\cdot \phi), I_{d-1}(B')\rangle$ where $B'$ is a certain  $(d-1)\times (n-2)$ submatrix of the Jacobian dual matrix $B.$

In addition, all the above results prove that the analytic spread is maximal.

A main drive of the present work is  to derive a most comprehensive landscape for condition $(*)$ in the case of a linear matrix $\phi$.
As such, for our purpose an important actor is the following notion introduced in \cite{Herzog_et_al2005}. Let $R$ be a standard graded ring over a field $k$, with irrelevant maximal ideal $\fm$,  and $I\subset R$  an equigenerated ideal thereof. Then the special fiber $\mathcal{F}_R(I)=\mathcal{R}_R(I)/\fm \mathcal{R}_R(I)$ can be identified with the $k$-subalgebra $k[It]\subset R[It]$.
In this way, the homogeneous defining ideal $Q$ of the latter is contained in a defining ideal of the Rees algebra $\mathcal{R}_R(I)$.
One then says that the ideal $I$ (or $\mathcal{R}_R(I)$) is of {\em fiber type} if a defining ideal of $\mathcal{R}_R(I)$ has the form $\langle I_1(\tt\cdot \phi), Q\rangle$.

In this regard, assuming throughout that $k$ is algebraically closed, we prove:

\begin{Theorem}\label{main_thm}
	Let $R$ and $I\subset R=k[x_1,\ldots,x_d]$ be as above, and such that $\mu(I) \geq d+1$ and $I_1(\phi)=\langle x_1,\ldots,x_d \rangle$.   If $I$ satisfies condition $G_{d-1},$ but not condition $G_d,$ then$:$
	\begin{enumerate}
		\item[\rm(i)]  The analytic spread $\ell(I)$ is maximal {\rm (}i.e., $\ell(I)=d${\rm )}.
		\item[\rm(ii)] The Rees algebra $\mathcal{R}_R(I)$ of $I$ is of fiber type, but not of the {\rm ``}expected\,{\rm ''} form. 
	\end{enumerate}
\end{Theorem}
The main contribution of this result is to collect the previous results under one and the same umbrella.
Of course, the question remains as to when  the special fiber is wholly defined by a determinantal ideal.

Concerning item (i)  we actually prove, more strongly, the following generalization of \cite[Theorem 2.4 (a)]{DRS}:

\begin{Proposition}\label{birationality} With the same notation as above, let $I_{n-1}$ denote the linear span of $I$.  
	Then the rational map $\psi:\pp_k^{d-1}\dasharrow \pp_k^{n-1}$ defined by $I_{n-1}$ is birational onto the image.
\end{Proposition}
Here, $\pp_k^{d-1}={\rm Proj}(R)$ denotes projective space over $k$ in homogeneous coordinates $x_1,\ldots,x_d$, while $\pp_k^{n-1}$ is similarly defined. The linear span $I_{n-1}$ is nothing else than the $k$-vector space spanned by the forms of degree $n-1$ in $I$. We emphasize that the rational map $\psi$ does not depend on the chosen vector basis.
Besides, it is closely related to the Rees algebra and the special fiber of $I.$ Namely, the Rees algebra $\mathcal{R}_R(I)$ corresponds to the  bihomogeneous coordinate ring of the graph of $\psi$, and  the special fiber $\mathcal{F}_R(I)$ corresponds to the homogeneous coordinate ring of the image of $\psi$
 (for details on the algebraic side of rational maps and their images, we refer to \cite{DHS}).

In regards to the Cohen--Macaulay property of the Rees algebra in item (ii), it may fail quite commonly even for $d=3$, as noticed in \cite[Remark 4.14]{DRS}.
Since the notion of an expected reduction number is related to this property, it is not clear what would be a replacement for such a notion over here. In any case, it will be clear that the reduction number is bounded below by $d-2$, and there are examples where this number is $d-1$.

Our essential procedure resorts to a specialization argument followed by a close scrutiny of a family of representatives of the inverse map to the above rational map $\psi$.
More exactly,  the specialization argument allows us to associate to the ideal $I$  an ideal $\overline{I}$  satisfying the hypotheses of \cite[Theorem 1.3]{MU}, while the chosen family of representatives of the inverse map to $\psi$ obtained  through the dual Jacobian  dual matrix $B$ as in  \cite[Theorem 2.18]{DHS} allows to control the generators of the defining ideal of the Rees algebra of $I$ through the generators of the defining ideal of the Rees algebra of $\overline{I}$.
A curiosity is that the birationality criterion of \cite[Theorem 2.18]{DHS} does not require that the ground field $k$ be algebraically closed, and yet the present proof of Proposition~\ref{birationality} involves this hypothesis.
As we have no strong reason to expect that the proposition fails otherwise, perhaps there is a missing argument that allows to {\em assume} that $k$ is algebraically closed throughout.

	\section{Notation and main tools}
	
	Let $R$ denote a Noetherian ring and let $E$ be a finitely generated $R$-module having rank.
	We recall the definitions of certain  graded algebras associated with the pair $(R,E)$. For further details, the reader is referred to \cite{GradedBook}.
	
	\begin{enumerate}
		\item[$\bullet$] The symmetric algebra  $\mathcal{S}_R(E)=\bigoplus_{i\geq 0}\mathcal{S}_{R,i}(E)$ -- the direct sum of the symmetric powers of $E.$
		\item[$\bullet$] The Rees algebra
		$\mathcal{R}_R(E)=\mathcal{S}_R(E)/\langle R{\rm -torsion}\rangle $.
		
		This notion of Rees algebra of a module is correct provided  $E$ has a rank, and  moreover in this case it retrieves the usual definition in the case of an ideal 
	as the direct sum of its powers.
		
	\item[$\bullet$]	If, moreover, $R$ is local with maximal ideal $\fm$ (or standard graded with maximal irrelevant ideal $\fm$), we consider the special fiber  (or fiber cone) $\mathcal{F}_{R}(E):= \mathcal{R}_R(E)/\mathfrak{m} \mathcal{R}_R(E).$
	\end{enumerate}

In order to present $\mathcal{S}_R(E)$ as a residue algebra of a polynomial ring, we pick a (finite) set of, say, $m$ generators of $E$ over $R$, and the corresponding finite free presentation of $E$ over $R$
\begin{equation}\label{presentation_of_E}
	R^n \stackrel{\phi}{\lar} R^m \lar E\rar 0.
\end{equation}
Then the universal property of the symmetric algebra implies the exact sequence of symmetric algebras
$$\mathcal{S}_R(R^n) \stackrel{\mathcal{S}_R\phi}{\lar} \mathcal{S}_R(R^m) \lar \mathcal{S}_R(E)\rar 0,$$
providing an isomorphism of $R$-modules $\mathcal{S}_R(E)\simeq R[\tt]/I_1(\tt\cdot \phi)$, where $R[\tt]$ is a polynomial ring over $R$ in $m$ variables $\tt$ which is graded isomorphic to $\mathcal{S}_R(R^m)$ as $R$-algebras.

A presentation of the Rees algebra is available drawing upon the above representation of $\mathcal{S}_R(E)$, namely, we 
have $\mathcal{R}_R(E)\simeq R[\tt]/\langle I_1(\tt\cdot \phi), \tau(E)\rangle$, where $\tau(E)$ stands for the $R$-torsion of $\mathcal{S}_R(E)$ lifted to $R[\tt]$.

In general the complementary generators afforded by $\tau(E)$ can be quite involved.
In the case where $R$ is a standard graded ring over a field $k$ and $I\subset R$ is an equigenerated ideal thereof, then the special fiber $\mathcal{F}_R(I)$ can be identified with the $k$-subalgebra $k[It]\subset R[It]$.
In this way, the homogeneous defining ideal $Q$ of the latter is contained in $\tau(I)$.

There is a couple of ways to define the condition $G_s$, originally introduced in \cite{AN} in the case of ideals. 
Thus, let $R$ denote a Noetherian ring and let $E$ be a finitely generated $R$-module having rank $e$.

\begin{Definition}\label{def_of_G_s}\rm 
	Given an integer $s\geq 1$,  $E$  satisfies  {\em condition $G_s$} if $\mu(E_{\wp})\leq \dim R_{\wp}+e-1$ for every
	${\wp}\in \Spec(R)$ with $\dim R_{\wp}\leq s-1$.
	
	Equivalently, letting 
	$$R^r\stackrel{\phi}{\lar} R^m \lar E \rar 0$$ 
	stand for a finite free presentation of $E$ over $R$, 
	$E$ satisfies  condition $G_s$ if 
	\begin{equation}\label{G_s_in_terms_of_minors}
		\Ht I_j(\phi)\geq m-j-(e-2), \; \text{\rm for} \; m-s-(e-2)\leq j\leq m-e,
	\end{equation}
	where $\Ht \mathcal{I}$ denotes the height of an ideal $\mathcal{I}$.
\end{Definition}

The definition as applied to an ideal with positive grade translates, respectively, into any of the requirements
\begin{equation}\label{G_s_for ideals}
	\mu(I_{\wp})\leq \Ht {\wp}, \; \text{\rm for} \;
	{\wp}\in \Spec(R) \; \text{\rm with} \; \Ht {\wp}\leq s-1.
\end{equation}
or
\begin{equation}\label{G_s_for ideals-bis}
	\Ht I_j(\phi)\geq m-j+1, \; \text{\rm for} \; m-s+1\leq j\leq m-1.
\end{equation}
Next we state two results involving the role of the condition $G_s$.
The first is the following  corollary of \cite[Theorem 5.1]{Tch}:

\begin{Proposition}\label{initial-degree}
	Let $R$ be a standard graded Cohen–Macaulay ring over a field and let
	$I\subset R$ denote an equigenerated perfect homogeneous ideal of codimension $2.$ 
	Set $Q$ as above for the homogeneous defining ideal of the special fiber $\mathcal{F}_R(I)$, and assume that $Q\neq \{0\}$. If $I$ satisfies condition $G_s$,  then ${\rm indeg}(Q) \geq s.$
\end{Proposition} 

The condition also implies a lower bound for the reduction number.
Namely, in particular, if $\mu(I)> \dim R$ and $I$ satisfies condition $G_s$, then $ \mu(JI^j) < \mu (I^{j+1})$ for all $0 \leq j\leq s-2$, 
where $J$ denotes a minimal reduction of $I$.
That is, the reduction number of $I$ is at least $s-1$. 

The second result has to do with the dimension of the symmetric algebra $\mathcal{S}_R(E)$, of which much has been said way back in regard to the behavior of the ideals of minors of the presentation matrix $\phi$ of $E$ (see \cite{SV1}, \cite{SV2}).
Here the emphasis is on modules of projective dimension one and on the $G_s$ condition.
As such, the following statement does not seem to have been recorded before in this precision.
It will be quite useful in the subsequent sections.
The idea of the proof itself has been employed before in the case of ideals (see, e.g., \cite[Theorem 4.1]{L1}).

\begin{Proposition}\label{dimSym}
	Let $R$ be a standard graded Cohen-Macaulay  ring of dimension $d$ and let $E$ be an $R$-module with rank $e$.
	Assume the following conditions$:$
	\begin{enumerate}
		\item[{\rm (i)}] $E$ has projective dimension one, with minimal free resolution
		$$0\to R^{n-e}\stackrel{\phi}\lar R^n\to E\to 0,$$
		such that $n\geq  d+e-1$.
		\item[{\rm (ii)}]  $E$ satisfies condition $G_s,$ for some $1\leq s\leq d-1.$
		\item[{\rm (iii)}] With $s$ as in item {\rm (ii)}, there is a prime ideal $I_1(\phi)\subset {\bf p} \subset R$ with $ \Ht {\bf p} \geq s $, such that  $\Ht I_{j}(\phi)\geq \Ht {\bf p}-j+1$ for every  $1\leq j\leq \Ht {\bf p}-s$. 
	\end{enumerate}
	Then  $\dim \mathcal{S}_R(E)=n+\dim R/{\bf p}.$ 
\end{Proposition}
\begin{proof}
	We apply the Huneke--Rossi formula (\cite{HuRo})
	$$\dim {\mathcal S}_R(E)= {\rm sup}_{\wp\in {\rm Spec} R} \left\{\dim R/\wp+ \mu(E_\wp)\right\},$$
	by sweeping through the possible values of $\Ht \wp$.
	
	\medskip
	
	$\bullet$   $\Ht \wp\leq s-1$
	\begin{eqnarray}
		\dim R/\wp+\mu(I_{\wp}) &=& d-\Ht \wp+\mu(E_{\wp}) \nonumber\\
		&\leq & d +e-1\quad\quad\quad\quad\quad\quad\quad (\mbox{because $E$ satisfies $G_s$})\nonumber\\
		&\leq &n \quad\quad\quad\quad\quad\quad\quad\quad\quad\quad\,\,(\mbox{by hypothesis}).
	\end{eqnarray}
	
	$\bullet$ $s\leq \Ht \wp <\Ht{\bf p}.$ 
	
	Set $j:=\Ht {\bf p}-\Ht \wp.$ In particular, $1\leq j\leq \Ht{\bf p}-s.$ So, by hypothesis we have $\Ht I_j(\phi)\geq \Ht {\bf p}-j+1=\Ht\wp+1.$ From this, $I_j(\phi)\not\subset \wp.$ Thus, $\mu(E_{\wp})\leq n-j=n-\Ht {\bf p}+\Ht \wp.$ 
	Therefore,
	\begin{equation}
		\dim R/\wp+\mu(E_{\bf p})\leq d-\Ht \wp+n-\Ht {\bf p}+\Ht \wp=n+d-\Ht{\bf p} =n+\dim R/{\bf p}
	\end{equation}
	
	$\bullet$ $\wp={\bf p}$
	
	Since $I_1(\phi)\subset{\bf p}$, then
	$$0\to R_{{\bf p}}^{n-e}\stackrel{\phi}\lar R_{{\bf p}}^n\to E_{{\bf p}}\to 0$$
	is a minimal free resolution of $E_{{\bf p}}$ over $R_{{\bf p}}.$ Therefore, $\mu(E_{{\bf p}})=n$ and
	$$\dim R/{\bf p}+\mu(E_{{\bf p}})=n+\dim R/{\bf p}.$$

	$\bullet$ $\Ht\wp\geq \Ht{\bf p}$
	
	In this case we have $\dim R/\wp\leq \dim R/{\bf p}$ and $\mu(E_{\wp})\leq n.$ Thus,
	$$\dim R/\wp+\mu(E_{\wp})\leq n+\dim R/{\bf p}$$
	
	Having swiped through all possible values  of $\Ht \wp$, we have $\dim \mathcal{S}_R(E)=n+\dim R/{\bf p}$ as claimed.
\end{proof}

\section{Preliminaries}\label{Prelims}
	
	We keep the assumptions and notation of the previous section,
	where $R=k[x_1,\ldots,x_d]$ is a standard graded polynomial ring and $I\subset R$ is a perfect ideal of codimension $2$ with an $n\times (n-1)$ linear presentation matrix $\phi$.
	In addition, we stipulate that throughout $I_1(\phi)=\langle x_1,\ldots,x_d \rangle$ and the ideal $I$ satisfies condition $G_{d-1}$, but not condition $G_d$.
	
	This section contains a lemmata on the main ingredients  used in the proofs of the two main results of the paper.
	
	The first lemma comprises elements of an extended notion of the chaos number as introduced in \cite{DRS} to arbitrary dimension $d\geq 3$.

	\begin{Lemma}\label{petit_everything}
		With these assumptions, 
		one has$:$
		\begin{enumerate} 
			\item[\rm (i)] $\Ht I_{n-d+1}(\varphi)=d-1$ {\rm (See also \cite[Lemma 2.4]{CPW1})}.
			\item[\rm (ii)]  There is a least integer $1\leq u\leq n-d$ such that $\Ht I_{u+1}(\varphi)=d-1$.
			Thus, $\Ht I_i(\phi)=d$ for every $i\leq u$.
			\item[\rm (iii)]  For any minimal prime $\wp$ of $R/I_{u+1}(\phi)$, the rank of $\phi$ over $R/\wp$ is $u$.
			\item[\rm (iv)]   $\dim \mathcal{S}_R(I)=n$ {\rm (Covers \cite[Theorem 4.1]{L1})}.
			\item[\rm (v)]  The Jacobian dual matrix $B$ of $\phi$ has rank $d$ over $k[\tt]$.
		\end{enumerate}
	\end{Lemma}

\begin{proof}
	(i)  Since $I$ satisfies $G_{d-1}$ then, by (\ref{G_s_for ideals-bis}), one has $\Ht I_{n-d+2}(\phi)\geq n-(n-d+2)+1=d-1$.
	Since $I_{n-d+2}(\phi)\subset I_{n-d+1}(\phi)$, then $\Ht I_{n-d+1}(\phi)\geq d-1$ as well.
	On the other hand, if $G_d$ fails at some $n-d+2\leq j\leq n-1$ then $G_{d-1}$ also fails at this value.
	Therefore, $G_d$ must fail at the value $j=n-d+1$, that is,
	$\Ht I_{n-d+1}(\phi)< n-(n-d+1)+1=d$.
	
	(ii) This is an immediate consequence of the first item. 
	
	(iii) Let $r$ denote the rank of $\phi$ over $R/\wp$, that is, $r=\min\{t\,|\,I_t(\varphi)\not\subset \wp\}.$
	Since $I_{u+1}(\varphi)\subset\wp$ and $I_{u}(\varphi)\not\subset \wp$ (because $\Ht I_{u}(\varphi)=d$) the  statement follows.

	(iv) We apply Proposition~\ref{dimSym} with $E=I$,  $s=d-1$, and ${\bf p}=\fm:=\langle x_1,\ldots, x_d\rangle$. The assumptions of that proposition are easily verified --  assumption (iii) follows due to having $\fm=I_1(\phi)$.
	
	
(v)	In order to show that the rank of $B$ is $d$, let $M$ be the cokernel of $B.$ Thus, we have the finite presentation
	$$k[\tt]^{n-1}\stackrel{B}\lar k[\tt]^d\to M\to 0.$$
We have $\mathcal{S}_R(I)\simeq \mathcal{S}_{k[\tt]}(M)$ as follows from the equality $\tt\phi=\xx B$. Thus, $\dim \mathcal{S}_{k[\tt]}(M)=n.$ Hence, once more by the Huneke-Rossi formula, localizing at the zero ideal:
	$$n+\mu(M_{(0)})\leq n.$$
	Therefore, ${\rm rk}\, M= \mu(M_{(0)})=0.$ Consequently, ${\rm rk}\, B=d-{\rm rk}\, M=d.$
\end{proof}

The next lemma follows the principle of \cite[The proof of Theorem 2.4]{DRS}, by setting $\phi$ in a more convenient form.

\begin{Lemma}\label{formating_phi}
	{\rm ($k$ algebraically closed)}
	The  matrix $\varphi$ is conjugate to a matrix of the following form
	\begin{equation}\label{canonical}
		\left[
		\begin{array}{ccc|cccccc}
			x_1+a_1 &  &  &  & && & &\\
			& \ddots &   &  & & & & &\\
			&  & x_1+a_{u} &  & & & & &\\
			\rule{1cm}{0,1pt} & 	\rule{1cm}{0,1pt} & 	\rule{1cm}{0,1pt} &  & & & & &\\ 
			&  &  &  & & & & &\\
			&  &  &  & & & & &
		\end{array}
		\right],
	\end{equation}
	where  the $a_i$'s and the blank entries are linear forms in $k[x_2,\dots,x_d]$. 
\end{Lemma}

\begin{proof}  Let  $\wp$ be a minimal prime of $R/I_{u+1}(\phi)$.
	Since $k$ is algebraically closed and $\Ht I_{u+1}(\phi)=d-1$, by a change of variables we may assume that $\wp=\langle x_2,\ldots, x_d\rangle$.
	
	Write
	$\varphi=x_1\varphi_1+\cdots+x_d\varphi_d,$
	where $\varphi_1,\ldots,\varphi_d$ are $n\times (n-1)$ matrices over $k.$ Note that $x_1\varphi_1$ is same as $\phi$ taken over $R/\wp.$ Thus, $u=\rank\varphi_1$ by Lemma~\ref{petit_everything} (3). Hence, $\phi_1$ is conjugate to  $$ \left[\begin{array}{cccc}\mathbb{I}_u&\boldsymbol{0}\\\boldsymbol{0}&\boldsymbol{0}\end{array}\right],$$
	by means of matrices $A\in {\rm GL}_{n}(k)$ on the left, and $B\in {\rm GL}_{n-1}(k)$ on the right. 
	Then $\phi$ is conjugate by these same matrices to one of the desired form.
\end{proof}
We assume throughout the rest of the paper that $\phi$ is as in \eqref{canonical}.

Then the corresponding Jacobian dual matrix $B$ acquires the following form:
\begin{equation}\label{B}
	B=\left[
	\begin{array}{cccc}
		t_{1} & \cdots & t_{u} & {\bf 0}\\
		&  &  & B'
	\end{array}
	\right],
\end{equation}
where $B'$ is a $(d-1)\times(n-1-u)$ matrix, and  the blank entries are linear forms in $k[\tt]$.

Let $\phi_u$ denote the submatrix of $\phi$  omitting the first $u$ columns, giving rise to a short exact sequence of $R$-modules
\begin{equation}\label{module_E}
	R^{n-1-u}\stackrel{\varphi_u}{\longrightarrow}R^{n}\to E\to 0.
\end{equation}

The next lemma concerns the main properties of this $R$-module.
We note that the first two items are also a consequence of \cite[Lemma 2.6]{CPW2}, but we give proofs for completeness.

\begin{Lemma}\label{rankB'} Let $E$ be as in {\rm (\ref{module_E})}. Then$:$
	\begin{enumerate}
		\item[\rm(i)]  $E$ has rank $u+1$ {\rm (}i.e., $E$ has projective dimension one{\rm )}.
		\item[\rm(ii)]  $E$ satisfies the $G_{d-1}$ condition.
		\item[\rm (iii)] $\dim {\mathcal S}_R(E)=n+1.$
	\end{enumerate}
\end{Lemma}

\begin{proof}
	(i)  Let $j\geq 1$ be an integer such that $j+u\leq n-1$.
	Any $(j+u)$-minor of $\phi$ involves at least $j$ columns of $\phi_u.$ Thus, by Laplace, every $(j+u)$-minor of $\phi$ is a linear combination over $R$ of $j$-minors of $\phi_u,$ that is, $I_{j+u}(\phi)\subset I_{j}(\phi_u).$
	
	On the other hand, suppose in addition that $n-(u+1)-(d-1)+2\leq j\leq n-(u+1),$ i.e., $n-(d-1)+1\leq j+u\leq n-1$.
	Since  $I$ satisfies condition $G_{d-1}$, then  ${\rm ht\,}I_{j+u}(\phi)\geq  n-(j+u)+1.$ Thus, ${\rm ht\,}I_{j}(\phi_u)\geq  n-(j+u)+1$ as well.
	In particular, for $j=n-u-1$, we find that $\Ht I_{n-u-1}(\phi_u)\geq 2.$
	This implies that ${\rm rk\,}E=u+1$.

	(ii) By  the argument in the previous item, now applying the criterion of (\ref{G_s_in_terms_of_minors}), with  ${\rm rk\,}E=u+1$, we have that  $E$ satisfies the $G_{d-1}$ condition.
	
	(iii) By the previous item, $E$ is an $R$-module with rank $e=u+1$ having minimal graded free resolution
	$$0\to R^{n-1-u}\stackrel{\varphi_u}{\longrightarrow}R^{n}\to E\to 0$$
	satisfying $G_{d-1}.$ We apply Proposition~\ref{dimSym} by verifying its three assumptions: 
	\begin{enumerate}
		\item[$\bullet$] $d+e-1\leq n$ because $u\leq n-d$ and $e=u+1;$
		\item[$\bullet$] $I_1(\phi_u)\subset {\fn} := \langle x_2,\ldots,x_d\rangle;$
		\item[$\bullet$] Since $\fn =(d-1)$, this assumption is emptily verified.
	\end{enumerate}
	Thus, $\dim \mathcal{S}(E)=n+\dim R/\fn=n+1.$
\end{proof}

Our last lemma concerns  the submatrix $B'$  of the Jacobian dual matrix $B$ as in (\ref{B}).

\begin{Lemma}\label{rankB'}
	The $(d-1)\times(n-1-u)$ matrix submatrix	$B'$ in {\rm (\ref{B})} has maximal rank.
\end{Lemma}
\begin{proof}
	Set $A:=\left[\begin{matrix}\boldsymbol0\\B'\end{matrix}\right]$ for lighter reading, where $\boldsymbol0$ is the $1\times(n-u-1)$ null matrix, and let $M$ denote the cokernel of the map of $k[\tt]$-modules  defined by $A$, namely: 
	$$k[\tt]^{n-u-1}\stackrel{A}\lar k[\tt]^d\to M\to 0.$$
	Clearly, ${\rm rk\,} M=d-{\rm rk}\,A=d-{\rm rk}\,B'\geq 1.$
	On the other hand, $A$ is a Jacobian dual like matrix relative to $\phi_u$, that is,
	$$\tt\phi_{u}=\xx A.$$
	Therefore, ${\mathcal S}_{k[\tt]}(M)\simeq {\mathcal S}_R(E).$ Thus, $\dim {\mathcal S}_{k[\tt]}(M)=n+1.$ Hence, once more by the Huneke-Rossi formula, localizing at the zero ideal:
	$$n+\mu(M_{(0)})\leq n+1,$$
	that is, ${\rm rk\,}M={\rm rk}\,M_{(0)}=\mu(M_{(0)})\leq 1.$ Therefore, ${\rm rk\,}M=1.$ In particular, ${\rm rk\,}B'=d-1.$ 
\end{proof}

\section{Proofs of the main results}

\subsection{Proof of Proposition~\ref{birationality}}\label{Birationality tools}

We restate the proposition for convenience.

\begin{Proposition} \label{birationality_bis}
	 The rational map $\psi:\mathbb{P}_k^{d-1}\dasharrow \mathbb{P}_k^{n-1}$ defined by the linear span $I_{n-1}$ of $I$ is birational onto the image.
\end{Proposition}
\begin{proof}
	Note that the image of $\psi$ is identified with ${\rm Proj}(\mathcal{F}_R(I))$.
	The proof follows pretty much the steps of the proof of \cite[Theorem 2.4]{DRS}, via the birationality criterion of \cite[Theorem 2.18]{DHS}. 
		Namely, let then $Q\subset k[\tt]$ denote the homogeneous defining ideal of $\mathcal{F}_R(I)$.
 We will show that there is a square submatrix $M$ of the matrix $B$ in \eqref{B}, of order $d-1$, such that $\det(M)\neq 0 \text{ (mod } Q)$, that is, ${\rm rk\,}_{\mathcal{F}_R(I)}(B)\geq d-1$. From this, the result follows by the above criterion.
	
	By Lemma~\ref{rankB'}, $I_{d-1}(B')\neq 0.$ Thus, since $I_{d-1}(B')\subseteq I_{d-2}(B'),$  there is a square submatrix $M'$ of $B'$ of order $d-2$ such that $\det(M')\neq 0.$ Let 
	$$M=\left[
	\begin{array}{cc}
		t_{n} & 0 \\
		* & M'
	\end{array}
	\right],$$
	a square submatrix of $B$ of order $d-1$. Clearly, $\det(M)=t_n\det(M')\neq 0$. 
	
	We assert that $\det(M)\not\subset Q$.
	Otherwise, since $Q$ is a prime, then either $t_u\in Q$ or else $\det(M')\in Q$, both of degrees strictly  less than $d-1$.
	But, since $I$ satisfies $G_{d-1}$,  Proposition~\ref{initial-degree} says that the initial degree of $Q$ is at least $d-1$.
	
	Thus, $\det (M)\neq\, 0 \,{\rm mod}\, Q$, and since $M$ is a submatrix of $B$, then ${\rm rk}_{{\mathcal F}_R(I)}(B)\geq d-1$ as was to be shown.
\end{proof}
Throughout, $Q$ will stand for the homogeneous defining ideal of $\mathcal{F}_R(I)$.

For any $d\times (d-1)$ submatrix $\mathcal{B}$ of the Jacobian dual matrix $B$  we   denote the sequence of the its $d$ ordered signed $(d-1)$-minors  by $\boldsymbol{\delta}_{\mathcal{B}}=\{\delta_\mathcal{B}^{1},\ldots,\delta_{\mathcal{B}}^{d}\}.$  The role of these sequences of minors  will  depend on the  rank of $\mathcal{B}$ modulo $Q.$ Thus, if the rank of $\mathcal{B}$ module $Q$ is less than $d-1$ then, obviously,  $\boldsymbol{\delta}_{\mathcal{B}}\subset Q.$  
Next is an overview of this role in the case where the rank is exactly $d-1.$

\begin{Corollary}\label{birationality_excerpts}
Set $R=k[x_1,\ldots,x_d]$, and	let $\psi$ be as in the previous proposition. The following statements hold$:$
	\begin{enumerate}
		\item[\rm (i)]  There is a $d\times(d-1)$ submatrix $\mathcal{B}$ of $B$ of rank $d-1$ modulo $Q$.
		\item[\rm (ii)] For every $d\times(d-1)$ submatrix $\mathcal{B}$ of $B$ of rank $d-1$ modulo $Q$,  the sequence of its ordered signed $(d-1)$-minors  gives a representative of the inverse to $\psi.$
		\item[\rm (iii)] For each matrix $\mathcal{B}$ as in the previous item, along with the sequence  $\boldsymbol{\delta}_{\mathcal{B}}=\{\delta_\mathcal{B}^{1},\ldots,\delta_{\mathcal{B}}^{d}\}$ of its $d$ ordered signed $(d-1)$-minors defining an inverse to $\psi$, there is an isomorphism  of $R[\tt]/Q$-algebras
		\begin{equation}\label{old_isom}
		\partial_{\mathcal{B}}: \mathcal{R}_R(I){\stackrel{\sim}{\rightarrow}} \mathcal{R}_{k[\tt]/Q}(\langle \boldsymbol{\delta}_{\mathcal{B}}, Q\rangle/\langle Q\rangle),
		\end{equation}
		whose restriction to $R=k[x_1,\ldots,x_d]$ maps $x_i\mapsto \delta_{\mathcal{B}}^i$ for every $1\leq i\leq d.$
	\end{enumerate}
\end{Corollary}
\begin{proof}
	(i) This  is essentially an outcome of the argument in the proof of Proposition~\ref{birationality_bis}.
	
	 (ii) This stems from  \cite[Theorem 2.18, (a) and Supplement (ii)]{DHS}.	
	
	(iii) This  is \cite[Proposition 2.1]{Sim}. Note, for further clarification, that $\mathcal{R}_R(I)\simeq R[\tt]/\mathcal{J}$, where $\mathcal{J}$ contains $Q$, so $\mathcal{R}_R(I)$ is naturally an $R[\tt]/Q$-algebra.
\end{proof}

\subsection{Proof of Theorem~\ref{main_thm}}

A main tool in this section is a specialization toward one dimension less.
With this in mind we choose a suitable form in $R=k[x_1,\dots,x_d]$ which is regular over $R/I$.

Namely, since $\Ht I_{n-d+1}(\varphi)= d-1$ (Lemma 3.1(1)) and  $I_{n-d+1}(\varphi)\supseteq I_{n-d+2}(\varphi)\supseteq\cdots\supseteq I_{n-1}(\varphi)=I$ we have $\Ht I_j(\varphi)\leq d-1$ for $n-d+2\leq j\leq n-1.$ Thus, by prime avoidance, there is linear form $\mathfrak{f}=x_1-b_2x_2-\cdots-b_dx_d\in k[x_1,\dots,x_d]_1$  avoiding the minimal primes of the minors ideal $I_j(\varphi)$ for $n-d+2\leq j\leq n-1.$  In particular,  $\mathfrak{f}$ is regular over $R/I$ and $\dim R/\langle I_{j}(\phi),\mathfrak{f}\rangle=\dim R/I_{j}(\phi)-1$ for every $n-d+2\leq j\leq n-1.$  

In particular, identifying $R/\langle \mathfrak{f}\rangle \simeq k[x_2,\dots,x_d]=:\overline{R}$, the natural projection $R\twoheadrightarrow R/\langle \mathfrak{f}\rangle$ maps $I$ onto the height $2$ perfect ideal   $\overline{I}\subset\overline{R}$ satisfying $G_{d-1}$, with an $n\times(n-1)$ syzygy (Hilbert--Burch)  matrix over $k[x_2,\dots,x_d]$ of the following  form

\begin{equation}\label{canonicalbis}
	\overline{\varphi}=\left[
	\begin{array}{cccccc}
		b_2x_2+\cdots+b_dx_d+a_1 & &  &&& \\
		& \ddots &  & &&\\
		&   & b_2x_2+\cdots+b_dx_d+a_u &&& \\
		&  &  &&&  \\
	\end{array}
	\right],
\end{equation}
where  the $a_i$'s and the blank entries are linear forms. Thus, by a similar reason, the corresponding Jacobian dual matrix -- namely, the unique $(d-1)\times (n-1)$ matrix $\overline{B}$ with linear entries in $k[\tt]$ such that $\tt\cdot\overline{\varphi}=[x_2\cdots x_d]\overline{B}$ -- has the following form:

{\small\begin{equation}\label{Bbarra-formato}
		\overline{B}=\left[\begin{array}{cccccc} 
			\ell_{2,1}+b_2\ell_{1,1} & \cdots & \ell_{2,u}+b_2\ell_{1,u} & \ell_{2,u+1}+b_2\ell_{1,u+1} & \cdots & \ell_{2,n-1}+b_2\ell_{1,n-1}\\
			\vdots & \ddots & \vdots & \vdots & \ddots & \vdots\\
			\ell_{d,1}+b_d\ell_{1,1} & \cdots & \ell_{d,u}+b_d\ell_{1,u} & \ell_{d,u+1}+b_d\ell_{1,u+1} & \cdots & \ell_{d,n-1}+b_d\ell_{1,n-1}
		\end{array}
		\right],
\end{equation}}
where $\ell_{i,j}$ is the $i,j$-th entry of $B,$ a linear form in $k[\tt]$. By \eqref{B}, $\ell_{1,j}=t_j$ if $1\leq j\leq u$ and $\ell_{1,j}=0$ if $u+1\leq j\leq n-1.$

Given a $d\times(d-1)$ submatrix $\mathcal{B}$ of $B$, let $\overline{\mathcal{B}}$ stand for the ``corresponding'' $(d-1)\times(d-1)$ submatrix of $\overline{B}$. 
That is to say, if $\mathcal{B}$ is formed with the columns of $B$ indexed by $v_1,\dots,v_{d-1}$ ($1\leq v_1<\cdots<v_{d-1}\leq n-1$),  then  $\overline{\mathcal{B}}$ is formed with the columns of $\overline{B}$  with same indices. 
By \eqref{Bbarra-formato}, for every $1\leq j\leq d-1$, the $j$th row of $\overline{\mathcal{B}}$  is the $(j+1)$th row of $\mathcal{B}$ plus $b_{j+1}$ times the first row of $\mathcal{B}.$

\begin{Lemma}\label{specialization} With the above  notation  along with {\rm Proposition~\ref{birationality_bis}}, one has$:$
	\begin{enumerate}
		\item[\rm(i)] The ideal $\langle I_1(\tt\cdot\overline{\varphi}), I_{d-1}(\overline{B})\rangle\subset \overline{R}[\tt]$ is a bihomogeneous defining ideal of  the Rees algebra of $\overline{I}$ over $ \overline{R}$.
		\item[\rm(ii)]  For every $d\times(d-1)$ submatrix $\mathcal{B}$ of $B$
		one has $\det(\overline{\mathcal{B}})=\mathfrak{f}(\delta^1_{\mathcal{B}},\ldots,\delta^d_{\mathcal{B}})$.
		In particular, if the rank of $\mathcal{B}$ modulo $Q$ is  less than $d-1$, then $\det(\overline{\mathcal{B}})\in Q.$ 
	\end{enumerate}	
\end{Lemma}

\begin{proof}
	(i) Since $\overline{I}$ satisfies $G_{d-1}$ and $k$ is infinite, this is the main part of Theorem~\ref{MainM-U}, (iii). 
	
	(ii) Write 
	$$\mathcal{B}=\left[\begin{matrix}
		L_1\\
		\vdots\\
		L_d
	\end{matrix}\right],$$ 
	where  $L_j=\left[\begin{matrix}\ell_{j,v_1}&\cdots&\ell_{j,v_{d-1}}\end{matrix}\right]$ for $1\leq j\leq d.$
	Then,  by \eqref{Bbarra-formato}
	\begin{eqnarray*}
		\det(\overline{\mathcal{B}})&=&\det\left[
		\begin{matrix}
			L_2+b_2L_1 \\
			\vdots  \\
			L_{d}+b_dL_1\\
		\end{matrix}
		\right]\\
		&=&\det\left[
		\begin{matrix}
			L_2 \\
			L_3\\
			\vdots   \\
			L_{d}\\
		\end{matrix}
		\right] +b_2\det\left[
		\begin{matrix}
			L_1 \\
			L_3\\
			\vdots\\
			L_d
		\end{matrix}
		\right]+
		b_3\det\left[
		\begin{matrix}
			L_2\\
			L_1\\
			\vdots\\
			L_{d}
		\end{matrix}
		\right] +\cdots+b_d\det\left[
		\begin{matrix}
			L_2 \\
			\vdots\\
			L_{d-1} \\
			L_1
		\end{matrix}
		\right]\\
		&=&
		\det\left[
		\begin{matrix}
			L_2 \\
			L_3\\
			\vdots   \\
			L_{d}\\
		\end{matrix}
		\right] +b_2\det\left[
		\begin{matrix}
			L_1 \\
			L_3\\
			\vdots\\
			L_d
		\end{matrix}
		\right]-
		b_3\det\left[
		\begin{matrix}
			L_1\\
			L_2\\
			\vdots\\
			L_{d}
		\end{matrix}
		\right] +\cdots+(-1)^{d-2}b_d\det\left[
		\begin{matrix}
			L_1\\
			L_2 \\
			\vdots\\
			L_{d-1} 
		\end{matrix}
		\right]\\
		&=&\delta^1_{\mathcal{B}}-b_2\delta^2_{\mathcal{B}}-b_3\delta^3_{\mathcal{B}}-\cdots-b_d\delta^d_{\mathcal{B}}\\
		&=&\mathfrak{f}(\delta^1_{\mathcal{B}},\ldots,\delta^d_{\mathcal{B}}).
	\end{eqnarray*}
	The complementary statement is a consequence of the fact that when the rank of $\mathcal{B}$ modulo $Q$ is less than $d-1$, then $\delta_{\mathcal{B}}^i\in Q$ for $1\leq i\leq d.$
\end{proof}

For convenience, in order to state again the main result, we retrieve the terminology and notation.
Let $R = k[x_1,\ldots, x_d]$  denote a standard graded polynomial ring over 
an algebraically closed  field $k$, and let $I \subset R$ be a perfect ideal of codimension $2$ with an $n\times (n-1)$ linear presentation matrix $\phi$ such that $n \geq d+1$ and $I_1(\phi)=\langle x_1,\ldots,x_d \rangle$.
Let $\mathcal{J}$ and $Q$ respectively denote the bihomogeneous defining ideal of the Rees algebra $\mathcal{R}_R(I)$  and the homogeneous defining ideal of the special fiber $\mathcal{F}_R(I)$, and let $B$ denote the associated Jacobian dual matrix.

Let $\{\mathcal{B}_1,\ldots ,\mathcal{B}_r\}$ be the total set of matrices as in Proposition~\ref{birationality_excerpts}(ii).
In order to simplify the notation, we write  $\boldsymbol{\delta}_i:=\boldsymbol{\delta}_{\mathcal{B}_i}$ for every $1\leq i\leq r,$ where $\boldsymbol{\delta}_{\mathcal{B}_i}=\{\delta_{\mathcal{B}_i}^1,\ldots,\delta_{\mathcal{B}_i}^{d}\}$ in a previous notation.
In addition, for $1\leq i\leq r$ we take liberty in denoting $\xx\mapsto{\boldsymbol{\delta}_i}$ as a mnemonic for the restriction to  $R=k[x_1,\ldots,x_d]$ of the isomorphism $\partial_{\mathcal{B}_i}$ of (\ref{old_isom}).

\begin{Theorem}
		   If $I$ satisfies condition $G_{d-1},$ but not condition $G_d,$ then$:$
	\begin{enumerate}
		\item[{\rm (i)}]   $\mathcal{J}=\langle I_1(\tt\cdot\varphi), Q\rangle$ as ideals of $R[\tt]$, i.e., $I$ is an ideal of fiber type.
		\item[{\rm (ii)}] The initial degree of $Q$ is $d-1;$ in particular, the inclusion of ideals
		$$\langle I_1(\tt\cdot\varphi), I_d(B)\rangle \subset \langle I_1(\tt\cdot\varphi), Q\rangle$$ 
		 is proper. Hence, the Rees algebra of $I$ is  not of the expected form.
	\end{enumerate}
\end{Theorem}

\begin{proof}
	(i) As discussed priorly, the goal is to prove that $\mathcal{J}\subseteq \langle I_1(\tt\cdot\varphi), Q\rangle$. 
	The strategy will consist in  inducting on the sum $p+q$, where $(p,q)$ is the bidegree of a bihomogeneous element $\mathbf{f}\in \mathcal{J}$.
	
	Clearly, $\mathbf{f}\in \langle I_1(\tt\cdot\varphi), Q\rangle$ for values $p+q\leq 2$, which we take as first step of the inductive procedure.
	
	Suppose then that $p+q\geq 3$. Consider the canonical commutative diagram below
	$$
	\begin{array}{ccccccc}
		0 \to & \mathcal{J} & \to & R[\tt] & \to & \mathcal{R}(I) & \to 0 \\
		& \downarrow &  & \downarrow &  & \downarrow & \\
		0 \to & \overline{\mathcal{J}} & \to & \overline{R}[\tt] & \to & \mathcal{R}_{\overline{R}}(\overline{I}) & \to 0 
	\end{array},
	$$
	where the downarrows are surjective. 
	
	Applying Lemma \ref{specialization} (i) yields $\overline{\mathbf{f}}\in 
	\langle I_{1}(\tt\cdot\overline{\varphi}),I_{d-1}(\overline{B})\rangle= \overline{\mathcal{J}}$.
	Say, $\overline{\mathbf{f}}=\overline{\mathbf{g}}+\overline{\mathbf{h}}$, where $\overline{\mathbf{g}}\in I_1(\tt\cdot\overline{\varphi})$ and $\overline{\mathbf{h}}\in I_{d-1}(\overline{B})$. 
	
	By Lemma~\ref{specialization} (ii), $\overline{\mathbf{h}}$ has a representative back in $\mathcal{J}\subset R[\tt]$ of the form
	$$	\mathbf{h}= \mathfrak{q}+ \mathfrak{c}_1(\xx,\tt)\mathfrak{f}(\boldsymbol\delta_1)+\cdots+\mathfrak{c}_r(\xx,\tt)\mathfrak{f}(\boldsymbol\delta_r),$$
	for certain forms $\mathfrak{c}_i(\xx,\tt)\in R[\tt],$ and where $\mathfrak{q}$ is a linear combination with coefficients in $R[\tt]$  of the elements $\mathfrak{f}(\boldsymbol{\delta}_{\mathcal{B}})$ such that   ${\rm rk\,}_{\mathcal{F}(I)}\mathcal{B}< d-1;$ in particular, by Lemma~\ref{specialization} (ii), the form $\mathfrak{q}$ belongs to $Q$. 
	
	Since minors commute with homomorphisms, lifting back to $R[\tt]$ through the above diagram yields 
	\begin{equation}\label{Fformat}
		\mathbf{f}-\mathbf{g}=\mathfrak{q}+ \mathfrak{c}_1(\xx,\tt)\mathfrak{f}(\boldsymbol\delta_1)+\cdots+\mathfrak{c}_r(\xx,\tt)\mathfrak{f}(\boldsymbol\delta_r)+\mathfrak{p}(\xx,\tt)\mathfrak{f},
	\end{equation}
	for a suitable bihomogeneous form $\mathfrak{p}(\xx,\tt)\in R[\tt].$ 
	
	But both $\mathbf{g}$ and $\mathfrak{q}$ belong to $\langle I_1(\tt\cdot\varphi), Q\rangle$, so it is enough to show that 
	$$\mathbf{f}_1:=\mathfrak{c}_1(\xx,\tt)\mathfrak{f}(\boldsymbol\delta_1)+\cdots+\mathfrak{c}_r(\xx,\tt)\mathfrak{f}(\boldsymbol\delta_r)+\mathfrak{p}(\xx,\tt)\mathfrak{f}\in \langle I_1(\tt\cdot\varphi), Q\rangle.$$

	At any rate, ${\bf f}_1\in\mathcal{J}.$
	Thus, for any $1\leq i\leq r$, applying the isomorphism $\partial_{\mathcal{B}_i}$ in (\ref{old_isom}) to ${\bf f}_1,$  lands out in $Q$.
	We now take $i=1$.
	Since the restriction of $\partial_{\mathcal{B}_1}$ is 
	 ${\bf x}\mapsto \boldsymbol{\delta}_1$, we get 
	$$\mathfrak{c}_1(\boldsymbol{\delta}_1,\tt)\mathfrak{f}(\boldsymbol\delta_1)+\mathfrak{c}_2(\boldsymbol{\delta}_1,\tt)\mathfrak{f}(\boldsymbol\delta_2)+\cdots+\mathfrak{c}_r(\boldsymbol{\delta}_1,\tt)\mathfrak{f}(\boldsymbol\delta_r)+\mathfrak{p}(\boldsymbol{\delta}_1,\tt)\mathfrak{f}(\boldsymbol{\delta}_1)\in Q.$$
	Now,  back by  the inverse map to $\partial_{\mathcal{B}_1}$:
	$$\mathfrak{c}_1(\xx,\tt)\mathfrak{f}+\underbrace{\mathfrak{c}_2(\xx,\tt)\mathfrak{f}(\boldsymbol\delta_2)+\cdots+\mathfrak{c}_r(\xx,\tt)\mathfrak{f}(\boldsymbol\delta_r)+\mathfrak{p}(\xx,\tt)\mathfrak{f}}_{=:{\bf f}_2(\xx,\tt)} \in\mathcal{J}.$$

	As $\mathbf{f}$ is bihomogeneous of bidegree $(p,q)$, its  summands have to be of bidegree $(p,q)$ as well.
	Thus, $\mathfrak{p}$ is bihomogeneous of bidegree $(p-1,q),$ and since a minor of $\boldsymbol{\delta}_i$ has degree $d-1$ in $\tt$, then
	$\mathfrak{c}_1(\xx,\tt),\ldots,\mathfrak{c}_r(\xx,\tt)$ are each bihomogeneous of bidegree $(p,q-d+1)$. 
	
	Moreover, as $\mathcal{J}$ is a bihomogeneous ideal,  and the respective bidegrees $(p,q)$ and $(p+1,q-d+1)$ of  $\mathbf{f}_2(\xx,\tt)$ and  $\mathfrak{c}_1(\xx,\tt)\mathfrak{f}$ are distinct, then both $\mathbf{f}_2(\xx,\tt)$ and  $\mathfrak{c}_1(\xx,\tt)\mathfrak{f}$ belong to $\mathcal{J}.$ In addition, $\mathfrak{f}\notin \mathcal{J},$ hence $\mathfrak{c}_1(\xx,\tt)\in\mathcal{J}.$ 
	
	Now, the sum of the bidegrees of $\mathfrak{c}_1(\xx,\tt)$ is $p+(q-d+1)=p+q-(d-1)<p+q$ since $d\geq 2$, hence $\mathfrak{c}_1(\xx,\tt)\in \langle I_1(\tt\cdot\varphi), Q\rangle$ by the inductive hypothesis. Then, using the fact that $\mathbf{f}_2(\xx,\tt)\in \mathcal{J}$ we can apply to it the isomorphism  $\partial_{\mathcal{B}_2}$  to conclude as before that $\mathfrak{c}_2(\xx,\tt)\in  \langle I_1(\tt\cdot\varphi), Q\rangle.$  In  this manner we eventually obtain that $\{\mathfrak{c}_1(\xx,\tt),\ldots,\mathfrak{c}_r(\xx,\tt)\}\subset \langle I_1(\tt\cdot\varphi), Q\rangle$, hence $\mathfrak{p}(\xx,\tt)\in \langle I_1(\tt\cdot\varphi), Q\rangle$ as well, as was to be shown.

	(ii) Consider the  following $d\times (n-u)$ submatrix  of $B$ 
	$$\tilde{B}=\left[\begin{matrix}t_u&\boldsymbol0\\\ast&B'\end{matrix}\right].$$ 
	By Lemma~\ref{rankB'}, $I_{d-1}(B')\neq 0$. Therefore, $I_{d}(\tilde{B})=t_uI_{d-1}(B')\neq 0$ as well.
	But, 
	$I_{d}(\tilde{B})\subset I_d(B)$, while the latter is contained in the bihomogeneous defining ideal  of $\mathcal{R}_R(I)$.
	Therefore, by item (i), one must have an inclusion $I_d(B)\subset Q$.  Since $Q$ is a prime ideal generated in degree $\geq d-1$, we must conclude that  $I_{d-1}(B')\subset Q$. This shows that the initial degree of $Q$ is $d-1$.
	In particular, as $I_d(B)$ lives in degree $d$, the inclusion  $\langle I_1(\tt\cdot\varphi), I_d(B)\rangle \subset \langle I_1(\tt\cdot\varphi), Q\rangle$ is proper. 
\end{proof}

			\bibliographystyle{amsalpha}

	\hspace{1cm} {\sc Thiago Fiel} \hspace{6.4cm} {\sc Zaqueu Ramos}\\
	Departamento de Matem\'atica, CCEN \hspace{2.5cm} Departamento de Matem\'atica, CCET\\
	Universidade Federal de Pernambuco  \hspace{2.5cm} Universidade Federal de Sergipe\\
	50740-560 Recife, PE, Brazil   \hspace{3.9cm}  49100-000 S\~ao Cristov\~ao, SE, Brazil\\
	thiagofieldacostacabral@gmail.com \hspace{2.88cm}  zaqueu@mat.ufs.br\\
	
	
	\begin{center}
		\noindent {\sc Aron Simis}\\
		Departamento de Matem\'atica, CCEN\\ 
		Universidade Federal de Pernambuco\\ 
		50740-560 Recife, PE, Brazil\\
		aron.simis@ufpe.br
	\end{center}

			\end{document}